\documentclass[12pt,reqno,tikz]{amsart}
\usepackage{amsopn,amsfonts,amssymb,amsthm,mathtools,amsmath}
\usepackage{tikz}
\usepackage{tikz-cd}
\usetikzlibrary{trees,arrows}
\usepackage[edges]{forest}
\usepackage{latexsym}

\usepackage[T1]{fontenc}
\usepackage{lmodern}
\usepackage{textcomp}
\usepackage[hidelinks]{hyperref}
\usepackage[utf8]{inputenc}
\usepackage{enumerate}
\usepackage[shortlabels]{enumitem}
\setlist[enumerate, 1]{\sc(1)}
\usepackage{float}
\usepackage{hhline}
\usepackage{multirow}
\usepackage{anysize}

\usepackage{enumerate}

\def\z{\mathfrak{z}}

\def\b{\mathfrak{b}}

\def\g{\mathfrak{g}}

\def\a{\mathfrak{a}}

\def\so{\mathfrak{so}}

\def\X{\mathfrak{X}}

\def\B{\mathcal{B}}

\def\l{\ell}

\def\C{\mathbb{C}}
\def\R{\mathbb{R}}
\def\Q{\mathbb{Q}}
\def\Z{\mathbb{Z}}
\def\N{\mathbb{N}}

\newcommand{\GL}{{\sf GL}}
\newcommand{\SL}{{\sf SL}}
\newcommand{\SU}{{\sf SU}}

\def\div{\operatorname{div}}
\def\d{\operatorname{d}}
\def\ad{\operatorname{ad}}

\def\I{\operatorname{I}}

\def\diag{\operatorname{Diag}}

\newcommand{\End}{\operatorname{End}}
\newcommand{\Aut}{\operatorname{Aut}}

\DeclareMathOperator{\Iso}{Iso}

\DeclareMathOperator{\hol}{Hol}

\def\alt{\raise1pt\hbox{$\bigwedge$}}

\def\la{\langle}
\def\ra{\rangle}
\def\multiset#1#2{\ensuremath{\left(\kern-.3em\left(\genfrac{}{}{0pt}{}{#1}{#2}\right)\kern-.3em\right)}}
\renewcommand{\d}[1]{\ensuremath{\operatorname{d}\!{#1}}}

\DeclareSymbolFont{extraup}{U}{zavm}{m}{n}

\theoremstyle{plain}
\newtheorem{teo}{\bf Theorem}[section]
\newtheorem{coro}[teo]{\bf Corollary}
\newtheorem{prop}[teo]{\bf Proposition}
\newtheorem{lema}[teo]{\bf Lemma}

\theoremstyle{definition}
\newtheorem{defi}[teo]{\bf Definition}

\theoremstyle{remark}
\newtheorem{obsx}[teo]{Remark}
\newenvironment{obs}
{\pushQED{\qed}\obsx}

\newcommand{\matriz}[1]{\ensuremath{\begin{pmatrix}#1\end{pmatrix}}}

\newcommand*{\derp}[3][]{\ensuremath{\frac{\partial^{#1} #2}{\partial #3}}}
\newcommand{\comillas}[1]{``#1''}

\title{$G_2$-structures on flat solvmanifolds}
\author{Alejandro Tolcachier}
\email{atolcachier@famaf.unc.edu.ar}

\address{FAMAF-CIEM (CONICET), Universidad Nacional de Cordoba, Av. Medina Allende s/n, Ciudad Universitaria, X5000HUA, Córdoba, Argentina.}
\thanks{The author was partially supported by CONICET, SECyT-UNC, and MATHAMSUD Regional Programme 21-MATH-06.}
\begin{document}
	
	\renewcommand{\bibname}{References}
	
	\begin{abstract}
	In this article we study the relation between flat solvmanifolds and $G_2$-geometry. First, we give a classification of 7-dimensional flat splittable solvmanifolds using the classification of finite subgroups of $\GL(n,\Z)$ for $n=5$ and $n=6$. Then, we look for closed, coclosed and divergence-free $G_2$-structures compatible with the flat metric on them. In particular, we provide explicit examples of compact flat manifolds with a torsion-free $G_2$-structure whose finite holonomy is cyclic and contained in $G_2$, and examples of compact flat manifolds admitting a divergence-free $G_2$-structure.
	\end{abstract}
	
	\maketitle

	
	\section{Introduction}
	A $G_2$-structure on a 7-dimensional manifold $M$ is a globally defined 3-form $\varphi$ which can be pointwise written as
	\[\varphi=e^{123}+e^{145}+e^{167}+e^{246}-e^{257}-e^{347}-e^{356},\] with respect to a suitable basis $\{e^1,\ldots,e^7\}$ of the cotangent space where $e^{ijk}$ denotes $e^i \wedge e^j \wedge e^k$. Such a 3-form $\varphi$ induces a Riemannian metric $g_\varphi$, a Hodge star $\star_\varphi$ and a volume form $\operatorname{vol}_\varphi$ on $M$.
	
	$G_2$-structures can be divided into classes, which are characterized by the expression of the exterior derivatives $d\varphi$ and $d\star_\varphi \varphi$ \cite{FG82}. A $G_2$-structure is called \textit{closed} if $d\varphi=0$ and \textit{coclosed} if $d\star_\varphi \varphi=0$.
	
	The intrinsic torsion of a $G_2$-structure $\varphi$ can be identified with the covariant derivative $\nabla^\varphi \varphi$, where $\nabla^\varphi$ is the Levi-Civita connection of $g_\varphi$. By a classical theorem of Fernández-Gray \cite{FG82}, $\nabla^\varphi \varphi$ vanishes if and only if  $\d\varphi=0$ and $\d\star_\varphi \varphi=0$. In this case the $G_2$-structure $\varphi$ on $M$ is called \textit{torsion-free}. 
	
	The importance of torsion-free $G_2$-structures comes both from its historical relevance and its nice topological properties. In 1955, Berger's classification theorem \cite{Berger} suggested that $G_2$ might possibly be the holonomy group of certain Riemannian 7-manifolds. However, not a single example of such a manifold had yet been discovered until 1984. The first examples of non-compact 7-manifolds with holonomy $G_2$ were constructed by Bryant \cite{Bry87}. Around 1994, Joyce \cite{Joyce} found the first examples in the compact case. He proved also that $(M,\varphi)$ is torsion-free if and only if $\operatorname{Hol}(g_\varphi)\subset G_2$, and in the compact case the equality holds if and only if $\pi_1(M)$ is finite \cite{Joyce}. Moreover, when the $G_2$-structure is torsion-free, the induced metric $g_\varphi$ is Ricci-flat. Thus, according to \cite{AK}, if $g_\varphi$ is homogeneous then $g_\varphi$ is flat. 
	
	One possible approach to find torsion-free $G_2$-structures is to construct a flow of $G_2$-structures which under certain conditions would converge to a torsion-free one. This approach was originally taken by Bryant when he introduced the Laplacian flow of closed $G_2$-structures \cite{Bry05}. Later, Karigiannis, McKay and Tsui introduced the Laplacian coflow for coclosed $G_2$-structures \cite{KMMP12}. These two flows share the property that the fixed points are precisely torsion-free $G_2$-structures in both cases. This highlights the importance of finding closed and coclosed $G_2$-structures. Another type of flows which have been considered recently are isometric flows of $G_2$-structures, that is, flows that preserve the metric, while modifying the $G_2$-structure (a survey of recent progress can be seen in \cite{Gri20p}). For instance, one can consider the evolution of the 3-form $\varphi$ via the equation  \begin{equation}\label{flow}
		\begin{cases}
			\derp{\varphi(t)}{t} =\iota_{\div T_{\varphi(t)} } (\star_{\varphi(t)} \varphi(t)) \\
			\varphi(0)=\varphi_0,
		\end{cases}
	\end{equation} 
where the vector field $\div T_{\varphi}$ is the \textit{divergence} of the full torsion tensor $T_\varphi$ (see \eqref{torsion} below). It is clear that $G_2$-structures with $\div T_\varphi=0$ are critical points of \eqref{flow}. It is known that closed and also coclosed $G_2$-structures satisfy $\div T_\varphi=0$ (see for instance \cite{Gri20p}). 

Our aim in this article is to study the existence of closed, coclosed, torsion-free and also divergence-free $G_2$-structures in the world of flat solvmanifolds. 

A solvmanifold is defined as a compact homogeneous space $\Gamma\backslash G$ of a simply connected
solvable Lie group $G$ by a discrete subgroup $\Gamma$. Solvmanifolds generalize the well known
nilmanifolds which are defined similarly when $G$ is nilpotent. Both nilmanifolds and solvmanifolds
have provided a large number of examples and counterexamples in differential
geometry. For instance, the first example of a symplectic manifold without Kähler structure,
the so-called “Kodaira–Thurston manifold”, is a four dimensional nilmanifold \cite{Th}.
However, many important global properties of nilmanifolds cannot be generalized to solvmanifolds,
and for this reason these manifolds are currently widely studied.

On the other hand, compact flat manifolds are well understood due to the three classical
Bieberbach’s theorems and they have been used to study different phenomena in
geometry. For instance, questions about isospectrality (see \cite{MiaRo} and the references
therein), Kähler flat metrics with holonomy in $\SU(n)$ \cite{Dekimpe}, among others. The class of flat solvmanifolds lies in the intersection between the two theories of solvmanifolds and compact flat manifolds, and thus provide a
nice interplay between them. Also, this class is rich enough to produce a diverse
collection of examples. We will focus on a particular class of flat solvmanifolds, namely the splittable ones, which have a certain structure that allows them to be classified in a systematic way.

In Section 3 we will follow an approach considered in \cite{Tol1,Tol2} by the author to classify the $n$-dimensional splittable flat solvmanifolds for $n\leq 6$. We imitate the ideas to classify 7-dimensional splittable flat solvmanifolds which will serve as very explicit examples to our purpose of studying $G_2$-geometry. The classification is divided in two cases, according to whether we start from an almost abelian Lie algebra $\R\ltimes \R^6$ or a non almost abelian Lie algebra $\R^2\ltimes \R^5$.

Section 4 is devoted to studying the existence of $G_2$-structures in the almost abelian and non almost abelian cases. In the former case, we find examples of compact flat manifolds equipped with a torsion-free $G_2$-structure satisfying that the holonomy group of the underlying metric is cyclic, finite and contained in $G_2$. In the latter case we prove that there are no closed $G_2$-structures, meanwhile all 7-dimensional splittable flat solvmanifolds admit $G_2$-structures which are coclosed and divergence-free, respectively.
 
 \medskip
 
\textsl{Acknowledgments:} Special thanks go to Jorge Lauret. This article originated from a suggestion of his. Also, I want to thank Agustín Garrone, Andrés Moreno and Henrique Sá Earp for very fruitful conversations. Thanks to my advisor Adrián Andrada for his careful reading of the article. Finally, I'm very grateful to the IMEC at UNICAMP for the warm hospitality during my visit.
	\section{Preliminaries}
	\subsection{Flat solvmanifolds}
	In \cite{Mi}, Milnor characterized those Lie groups which admit a flat left invariant metric and he showed that they are all solvable of a very restricted form, proving that its Lie algebra decomposes orthogonally as an abelian subalgebra and an abelian ideal, where the action of the subalgebra on the ideal is by skew-symmetric endomorphisms. Such a Lie group equipped with a flat left invariant metric $(G,\la\cdot,\cdot\ra)$ will be called a flat Lie group and  $(\g,\la\cdot,\cdot\ra_e)$ will be called a flat Lie algebra.
	
	Using Milnor's characterization, Barberis, Dotti and Fino decompose further a flat Lie algebra in the following way. 
	
	\begin{teo}\cite[Proposition 2.1]{BDF}\label{alglieplanas}
		Let $(\g,\la\cdot,\cdot\ra_e)$ be a flat Lie algebra. Then $\g$ splits as an orthogonal direct sum, \[\g=\b\oplus\z(\g)\oplus [\g,\g]\] where $\b$ is an abelian subalgebra, $[\g,\g]$ is abelian and the following conditions are satisfied:
		
		\begin{enumerate}
			\item $\ad:\b\to\mathfrak{so}([\g,\g])$ is injective,
			
			\item $\dim [\g,\g]$ is even, and 
			
			\item $\dim\b\leq\frac{\dim[\g,\g]}{2}$.
			
		\end{enumerate}
	\end{teo}
	As a consequence, $\{\ad_X\mid X\in\b\}$ is an abelian subalgebra of $\so([\g,\g])$ and therefore, it is contained in a maximal abelian subalgebra. Since these are all conjugate, there exist an orthonormal basis $\B$ of $\z(\g)\oplus[\g,\g]$ and $\lambda_1,\ldots,\lambda_n\in \b^*$ such that for $X\in\b$,  
	\begin{equation}\label{adplana}
		[\ad_X]_\B=\matriz{0_s &&&&&\\&0&-\lambda_1(X)&&&\\&\lambda_1(X)&0&&&\\&&&\ddots&&\\&&&&0&-\lambda_n(X)\\&&&&\lambda_n(X)&0},
	\end{equation}
	where $n=\frac{\dim[\g,\g]}{2}$ and  $s=\dim\z(\g)$.
	
	Note that a flat Lie algebra $(\g,\la\cdot,\cdot\ra_e)$ is $2$-step solvable, since $[\g,\g]$ is abelian, and unimodular, since $\ad_X$ is skew-symmetric for all $X\in\b$. It also follows that the nilradical of $\g$ is $\z(\g)\oplus [\g,\g]$.
		
	We are interested in discrete subgroups $\Gamma$ of a flat simply-connected Lie group $G$ such that $\Gamma\backslash G$ is compact. This space endowed with the flat metric induced from $G$ is a compact flat manifold. 
	
	In general, if $G$ is a simply-connected solvable Lie group, a discrete and cocompact subgroup $\Gamma$ of $G$ is called a \textit{lattice} and the quotient $\Gamma\backslash G$ is called a \textit{solvmanifold}. With this definition, solvmanifolds are always compact, orientable, and parallelizable.
	
	It is well known that every simply-connected solvable Lie group $G$ is diffeomorphic to $\R^n$ for $n=\dim G$. Moreover, $\pi_1(\Gamma\backslash G)\cong\Gamma$. 
	
	The fundamental group of a solvmanifold plays an important role. Indeed, Mostow \cite{Mostow} proved that two solvmanifolds with isomorphic fundamental groups are diffeomorphic. 
	
	Since a flat solvmanifold $\Gamma\backslash G$ is, as mentioned before, a compact flat manifold, its fundamental group is isomorphic to a discrete torsion-free and cocompact subgroup of isometries of $\R^m$ with $m=\dim(\Gamma\backslash G)$. These subgroups are called ($m$-dimensional) \textit{Bieberbach groups} and are well described by the three classical theorems known as \comillas{Bieberbach's theorems}.  
	
	A purely algebraic characterization of the Bieberbach groups independent of their embedding into $\Iso(\R^m)$ was given by Zassenhaus \cite{Zassenhaus}.
	
	\begin{teo}
		An abstract group $\Gamma$ is isomorphic to an $n$-dimensional Bieberbach group if and only if $\Gamma$ contains a finite index, normal, free abelian subgroup $\Lambda$ of rank $n$, that is also maximal abelian.
	\end{teo}

	The subgroup $\Lambda$ is the unique normal maximal abelian subgroup of $\Gamma$ and is called the \textit{translation group of} $\Gamma$. In other words, a Bieberbach group $\Gamma$ satisfies an exact sequence \[1\to \Lambda\xrightarrow[]{\iota} \Gamma\xrightarrow[]{\pi} H\to 1,\] where $H=\Gamma/\Lambda$ is a finite group and $\operatorname{rank} \Lambda=m$. It is well known that the group $H$ can be identified with the Riemannian holonomy group of the compact flat manifold whose fundamental group is $\Gamma$ (see for instance \cite{Ch}). 
	
	\medskip
	
	We will focus in a special class of flat solvmanifolds, namely the splittable ones.
	
	A simply-connected solvable Lie group $G$ is called \textit{splittable} if it is isomorphic to $\R^k\ltimes_\phi N$ where $N$ is the nilradical of $G$ and $\phi:\R^k\to\Aut(N)$ is an homomorphism. A lattice $\Gamma$ of a splittable Lie group $\R^k\ltimes_{\phi} N$ will be called \textit{splittable} if it can be written as $\Gamma=\Gamma_1\ltimes_{\phi} \Gamma_2$ where $\Gamma_1\subset\R^k$ and $\Gamma_2\subset N$ are lattices of $\R^k$ and $N$ respectively. Consequently $\Gamma\backslash G$ will be called a \textit{splittable solvmanifold}. According to \cite{Bock}, when $k=1$ and $N\simeq \R^n$ every lattice is splittable. The Lie groups $\R\ltimes_\phi \R^n$ are called \textit{almost abelian}.
	
	The next theorem, which is a particular case of a more general theorem proved in \cite{Y}, gives a criterion to determine the splittable lattices in a splittable Lie group $G=\R^k\ltimes_{\phi} \R^m$.

	\begin{teo}\label{lattices}
	Let $G=\R^k \ltimes_{\phi} \R^{m}$ be a splittable Lie group, where $\R^m$ is the nilradical of $G$. Then $G$ has a splittable lattice if and only if there exists a basis $\{X_1,\ldots,X_k\}$ of $\R^k$  such that $\exp(\ad_{X_i})$ is similar\footnote{Throughout this article, a $n\times n$ matrix $A$ will be said to be \textit{similar} (or \textit{conjugated}) to $B$ if there exists $P\in \GL(n,\R)$ such that $P^{-1}AP=B$ and \textit{integrally similar} if $P\in \GL(n,\Z)$.} to an integer matrix for all $1\leq i\leq k$. In this case, the lattice is $\Gamma=(\bigoplus_{i=1}^k  \Z X_i) \ltimes_{\phi} P \Z^m$ where $P^{-1} \exp(\ad_{X_i}) P$ is an integer matrix.  
\end{teo}
	Denoting $E_i=P^{-1} \exp(\ad_{X_i}) P$, the lattice $\Gamma=(\bigoplus_{i=1}^k \Z X_i) \ltimes_{\phi} P \Z^{m}$ is isomorphic to the group $\Sigma_{E_1,\ldots,E_k}:=\Z^k\ltimes_{E_1,\ldots,E_k} \Z^{m}$, whose multiplication is given by \[(r,t)\cdot (r',t')=\left(r+r',t+E_1^{r_1}\cdots E_k^{r_k} t'\right),\; r=(r_1,\ldots,r_k), \;r'\in \Z^k,\, t,t'\in \Z^m.\]
	Note that the multiplication is well defined because $E_i E_j=E_j E_i$ for all $i,j$.
	
	A flat Lie algebra $\g=\b\oplus\z(\g)\oplus [\g,\g]$ can be written as $\g=\R^k\ltimes_{\ad} \R^{s+2n}$, where $\b\simeq \R^k$ and the nilradical is given by $\z(\g)\oplus [\g,\g]\simeq \R^{s+2n}$. The corresponding simply-connected flat Lie group can be written as $G=\R^k\ltimes_\phi \R^{s+2n}$ where $\phi( \sum_{i=1}^k s_i X_i)=\prod_{i=1}^k \exp(s_i \ad_{X_i})$, with $\{\ad_{X_i}\}_{i=1}^k$ as in \eqref{adplana}. Therefore a flat Lie group is a splittable Lie group.
	
	To classify the splittable flat solvmanifolds, we have to classify the splittable lattices of flat Lie groups (up to isomorphism, by Mostow's theorem).
	
	The next results, proved in \cite{Tol2}, show some sort of relation between the splittable lattices of flat Lie groups $G=\R^k\ltimes \R^{s+2n}$ and the finite abelian subgroups of $\GL(s+2n,\Z)$.
	
	\begin{prop}\label{hol}
		Let $G=\R^k\ltimes_\phi\R^{s+2n}$ be a splittable flat Lie group and $\Gamma$ a splittable lattice given by $\Gamma=(\bigoplus_{i=1}^k \Z X_i)\ltimes_{\phi} P \Z^{s+2n}$, where $E_i:= P^{-1}\exp(\ad_{X_i})P$ is integer for $1\leq i\leq k$. Then $\hol(\Gamma\backslash G)\cong \la E_1,\ldots, E_k\ra$.
	\end{prop}
In particular, the holonomy group of a flat almost abelian solvmanifold is (finite) cyclic (see also \cite[Theorem 3.7]{Tol1}).

Two conjugate subgroups of $\GL(s+2n,\Z)$ which can be obtained as the holonomy group of a flat solvmanifold give rise to isomorphic lattices, as the next lemma shows.

\begin{lema}
	Let $E_1,\ldots,E_k,F_1,\ldots,F_k\in \GL(s+2n,\Z)$ be commuting matrices of finite order. If $\la E_1,\ldots, E_k\ra$ is conjugate to $\la F_1,\ldots,F_k\ra$ in $\GL(s+2n,\Z)$ then $\Sigma_{E_1,\ldots,E_k}\cong \Sigma_{F'_1,\ldots,F'_k}$ for some generating set $\{F'_i\}_{i=1}^k$ of $\la F_1,\ldots, F_k\ra$. Furthermore, suppose that the cardinal of a minimal generating set of $\la E_1,\ldots, E_k\ra$ is $\l <k$. Then $\Z^k \ltimes_{E_1,\ldots,E_k} \Z^{s+2n} \cong \Z^\l \ltimes_{H'_1,\ldots,H'_\l} \Z^{s+2n+k-\l}$, where $H'_i=\matriz{\I_{k-\l}&\\& H_i}$ and $\{H_i\}_{i=1}^\l$ is a generating set of $\la E_1,\ldots, E_k\ra$. 	
\end{lema}
In conclusion, to determine all the isomorphism classes of splittable lattices, we must first look at the finite abelian subgroups of $\GL(n,\Z)$ up to conjugacy and see which of them can be obtained as the holonomy group of a flat solvmanifold. Then, we have to distinguish the lattices. A classification of the finite subgroups of $\GL(n,\Z)$  for $n\leq 6$ was obtained (for $n=5,6$ with aid of CARAT, see \cite{Pl}). A list of these subgroups can be found in  https://www.math.kyoto-u.ac.jp/$\sim$yamasaki/Algorithm/RatProbAlgTori/crystdat.html.

\medskip

\subsection{$G_2$-structures}
\begin{defi}
	Let $M$ be a 7-dimensional differentiable manifold. A $G_2$-\textit{structure} on $M$ is a 3-form $\varphi\in \Omega^3(M)$ such that, at every $p\in M$, there exists a basis $\{e_1,\ldots,e_7\}$ of $T_p M$ with respect to which \begin{equation}\label{phi1}
		\varphi_p=e^{123}+e^{145}+e^{167}+e^{246}-e^{257}-e^{347}-e^{356},
	\end{equation} 
	where $\{e^i\}_{i=1}^7$ is the dual basis of $\{e_i\}_{i=1}^7$ and $e^{ijk}$ denotes $e^i \wedge e^j \wedge e^k$. Such a 3-form $\varphi$ is called \textit{positive}.
\end{defi}
\begin{obs}
	The fact that $\varphi\in \Omega^3(M)$ is positive is equivalent to $\varphi$ being in the orbit $\GL(7,\R)\cdot \varphi_0$, where $\varphi_0\in\Omega^3(M)$ is defined pointwise by $(\varphi_0)_p=e^{127}+e^{347}+e^{567}+e^{135}-e^{146}-e^{236}-e^{245}$ and the action $\cdot : \GL(7,\R)\times \Omega^3(M)\to \Omega^3(M)$ is defined by \[ h \cdot \varphi(X,Y,Z)=\varphi(h^{-1} X,h^{-1} Y, h^{-1} Z)\quad \forall X,Y,Z\in\X(M).\]
	It is well known that the isotropy group $\{A\in \GL(7,\R)\mid A\cdot \varphi_0=\varphi_0\}$ is isomorphic to the exceptional 14-dimensional Lie group $G_2$.
	
	We point out that the 3-form $\varphi_0$ is widely used as the definition of a $G_2$-structure, but for our purposes $\varphi$ as in \eqref{phi1} will be more useful.
	
\end{obs}

The existence of a $G_2$-structure is entirely a topological question. While not all smooth 7-manifolds admit $G_2$-structures, there are many that do and they are completely characterized by the following proposition proved in \cite{Law}.

\begin{prop}\label{existence}
	A smooth 7-manifold $M$ admits a $G_2$-structure if and only if $M$ is both orientable and spinnable\footnote{A \textit{spin manifold} is an oriented Riemannian manifold with a spin structure on its tangent bundle.}.
\end{prop}

A $G_2$-structure $\varphi$ on a manifold $M$ gives rise to a Riemannian metric $g_\varphi$ with volume form $\operatorname{vol}_\varphi$ via the identity
\begin{equation}\label{gphi}
	g_\varphi(X,Y)\operatorname{vol}_\varphi=\frac{1}{6}\iota_X \varphi \wedge \iota_Y \varphi \wedge \varphi,\quad X,Y\in \X(M).
\end{equation}

The existence of a $G_2$-structure $\varphi$ on $M$ determines a decomposition of the space of forms on $M$ into irreducible $G_2$-representations. The space $\Omega^k:=\Omega^k(M)$ is irreducible if $k=0,1,6,7$. The spaces of 2-forms and 3-forms decompose as \[ \Omega^2=\Omega^2_7\oplus \Omega^2_{14}, \quad \Omega^3=\Omega_1^3\oplus \Omega_7^3\oplus \Omega_{27}^3,\] where each $\Omega_{\l}^k$ has (pointwise) dimension $\l$ and this decomposition is orthogonal with respect to the metric $g_\varphi$. The spaces $\Omega_7^2$ and $\Omega_7^3$ are both isomorphic to the cotangent bundle $\Omega_7^1=T^* M$. In \cite{Kar}, Karigiannis gives explicit isomorphisms between the space $\Omega_{14}^2$ and the Lie algebra $\g_2$ and between $\Omega_{27}^3$ and the space of traceless symmetric 2-tensors $\operatorname{Sym}_0^2 (T^* M)$ on $M$. The first identification comes from the canonical isomorphism between $\Omega^2$ and $\so(7)$, the second one is given by the maps 
\begin{eqnarray*}
 \iota:\operatorname{Sym}^2_0(T^* M)\to \Omega_{27}^3,& \quad \jmath:\Omega_{27}^3 \to \operatorname{Sym}_0^2 (T^* M) \\
 (\beta_{ij})\mapsto \sum_\l \beta_{ij} g_\varphi^{j\l} e^i \wedge \iota_{e_\l} \varphi,& \qquad\qquad\qquad\qquad  \tau\mapsto \jmath(\tau)(v,w)=\star_\varphi (\iota_{v} \varphi \wedge \iota_{w} \varphi \wedge \tau)
 \end{eqnarray*}   

 The decompositions $\Omega^4=\Omega_1^4 \oplus \Omega_7^4\oplus \Omega_{27}^4$ and $\Omega^5=\Omega_7^5\oplus \Omega_{14}^5$ are obtained by taking the Hodge star of the decompositions of $\Omega^3$ and $\Omega^2$, respectively.
 
 Applying this decomposition to $\d\varphi$ and $\d\star_\varphi\varphi$ gives the following definition.
 \begin{defi}
 	Let $\varphi$ be a $G_2$-structure on a 7-manifold $M$. Then there are unique forms $\tau_0 \in\Omega^0, \tau_1\in\Omega_7^1, \tau_2\in\Omega_{14}^2$ and $\tau_3\in \Omega_{27}^3$, called the \textit{torsion forms} of $\varphi$, such that 
 	\[
 		\d\varphi=\tau_0 \star_\varphi \varphi+3\tau_1 \wedge \varphi+\star_\varphi \tau_3,\;\; \text{and}\; \; \d\star_\varphi\varphi=4\tau_1 \wedge \star_\varphi\varphi +\star_\varphi \tau_2.
 	\]
 \end{defi}

The torsion forms can be explicitly computed from $\varphi$ and $\star_\varphi\varphi$ by means of the following identities:
\begin{eqnarray}\label{tau3}  \tau_0=\frac{1}{7} \star_\varphi (\d\varphi \wedge \varphi) , & \quad 
	\tau_1=-\frac{1}{12} \star_\varphi ( \star_\varphi  \d\varphi \wedge \varphi),  \\ \nonumber
	\tau_2=- \star_\varphi \d\star_\varphi \varphi+ 4 \star_\varphi (\tau_1 \wedge \d\star_\varphi \varphi), &\quad
	\tau_3= \star_\varphi \d\varphi-\tau_0 \varphi-3 \star_\varphi (\tau_1\wedge \varphi).
\end{eqnarray}

Moreover, the torsion forms are completely encoded in the \textit{full torsion tensor} $T_\varphi$ which is the $(0,2)$-tensor defined by
\begin{equation}\label{torsion} 
	T_\varphi=\frac{\tau_0}{4} g_\varphi-\star_\varphi(\tau_1\wedge \star_\varphi\varphi)-\frac{1}{2}\tau_2-\frac{1}{4} \jmath(\tau_3).
	\end{equation} Contracting with the metric, $T_\varphi$ can be seen as $T_\varphi\in \End(TM)$ and the expression above is expressed in terms of the irreducible $G_2$-decomposition $\End(TM)=W_0\oplus W_1\oplus W_2\oplus W_3$, where $W_0\simeq \Omega^0, W_1\simeq \Omega_7^3, W_2\simeq \Omega_{14}^2$ and $W_3\simeq \Omega^3_{27}$, see e.g. \cite{FG82}. The endomorphism $T_\varphi \in \End(TM)$ satisfies $\nabla_X\varphi= \iota_{T_\varphi(X)} \star_\varphi \varphi$.

Since the torsion $T_\varphi$ decomposes into four independent components, each component can be zero or nonzero. This gives $16$ distinct classes of $G_2$-structures, called \textit{Fernández-Gray} classes. Some relevant classes with their names are given in the following table:

\medskip

\begin{tabular}{|c|c|c|}
	\hline
	
	Name & Conditions & Torsion forms \\
	\hline
	 \textit{Closed}     &     $\d\varphi=0$  & $\tau_0=\tau_1=\tau_3=0$  \\
	\textit{Coclosed}      &     $\d\star_\varphi\varphi=0$  & $\tau_1=\tau_2=0$  \\
	\textit{Coclosed of pure type} & $\d \star_\varphi \varphi=0, \d\varphi \wedge \varphi=0$ & $\tau_0=\tau_1=\tau_2=0$ \\
	\textit{Locally conformal parallel} & $\d\varphi=3\tau_1\wedge \varphi$, $\d\star_\varphi\varphi=4\tau_1\wedge \star_\varphi\varphi$ & $\tau_0=\tau_2=\tau_3=0$\\
	\textit{Nearly parallel}      &   $\d\varphi=\lambda \star_\varphi \varphi$ ($\lambda\neq 0$) & $\tau_1=\tau_2=\tau_3=0$  \\
	\textit{Torsion-free}
	      &   $\d\varphi=0$ and $\d\star_\varphi\varphi=0$ & $\tau_0=\tau_1=\tau_2=\tau_3=0$ \\ [0ex] \hline
\end{tabular}

\bigskip

We can define a $G_2$-structure on any real 7-dimensional Lie algebra $\g$ with basis $\{e_i\}_{i=1}^7$ as a 3-form $\varphi_0\in \alt^3 \g^*$ in the form of \eqref{phi1}. This structure on the Lie algebra gives rise to a left invariant $G_2$-structure on the corresponding Lie group. Therefore, any 7-dimensional Lie group has a left invariant $G_2$-structure. Note that if this left invariant $G_2$-structure is torsion-free, then the left invariant metric $g_\varphi$ is flat since it is Ricci-flat \cite{AK}.

Given a left-invariant $G_2$-structure $\varphi$ on a solvable Lie group $G$ which admits a lattice $\Gamma$, we can naturally define a $G_2$-structure $\tilde{\varphi}$ in the solvmanifold $\Gamma\backslash G$ as follows: \begin{equation}\label{phisolv} \tilde{\varphi}_{\pi(p)}(u,v,w)=\varphi_p((d\pi)_p^{-1} u, (d\pi)_p^{-1} v, (d\pi)_p^{-1} v),\quad p\in G, \, u,v,w\in T_{\pi(p)} (\Gamma\backslash G).
\end{equation} The $G_2$-structure $\tilde{\varphi}$ will be called an \textit{invariant} $G_2$-structure.
	
Given a solvmanifold $\Gamma\backslash G$ with an invariant $G_2$-structure $\tilde{\varphi}$ defined as in \eqref{phisolv}, it is easily seen that the conditions in the table above are satisfied by $\tilde{\varphi}$ if and only if they are satisfied by the 3-form $\varphi$ defined at the Lie algebra level. 

As a corollary of Proposition \ref{existence} and the existence of invariant $G_2$-structures on a solvmanifold we have 
\begin{coro}
	Any 7-dimensional solvmanifold admits a spin structure.
\end{coro}
In particular, any flat solvmanifold admits a spin structure, and thus we obtain many examples of spinnable compact flat manifolds, which are interesting according to \cite{sz}.

\section{Classification of 7-dimensional splittable flat solvmanifolds}

The goal of this section is to classify 7-dimensional splittable flat solvmanifolds. We will follow the method given in \cite{Tol2}, which we described in the last part of the preliminaries of flat solvmanifolds.

Let $\g$ be a non-abelian 7-dimensional flat Lie algebra. According to Theorem \ref{alglieplanas} there are two possibilities for $\dim \b$, namely $\dim \b=1$ or $\dim \b=2$. If $\dim \b=1$ then $\g$ is almost abelian and if $\dim \b=2$ then $\g$ is not almost abelian.

\subsection{Almost abelian case}

A 7-dimensional almost abelian flat Lie algebra can be written as $\g=\R x\ltimes_{\ad_x}\R^6$ where $\ad_x$ can be written in some basis $\mathcal{B}$ of $\z(\g)\oplus [\g,\g]$ as the block matrix\footnote{Throughout the article we will denote  the block diagonal matrix $\matriz{A&0\\0&B}$ by $A\oplus B$.} \[[\ad_x]=\matriz{0&-a\\a&0}\oplus \matriz{0&-b\\b&0}\oplus \matriz{0&-c\\c&0},\quad a^2+b^2+c^2\neq0,\]

The corresponding Lie group is 
$G=\R\ltimes_{\phi} \R^6$ with \begin{equation}\label{phi2}
	\phi(t)=\matriz{\cos(at)&-\sin(at)\\ \sin(at)& \cos (at)} \oplus \matriz{\cos(bt)&-\sin(bt)\\ \sin(bt)&\cos(bt)}\oplus\matriz{\cos(ct)&-\sin(ct)\\ \sin(ct)&\cos(ct)}
\end{equation}
Next we find the values of $at_0, bt_0, ct_0$ such that $\phi(t_0)$ is similar to an integer matrix so that, according to Theorem \ref{lattices}, we obtain lattices. Note that if we change $at_0$ by $2\pi k \pm at_0$ we will get a similar matrix to $\phi(t_0)$ so the corresponding lattices will be isomorphic. Taking this into account, we have 
\begin{teo}\label{valuest0}
	Let $G=\R\ltimes_{\phi} \R^6$ with $\phi(t)$ as in \eqref{phi2}. Then $\phi(t_0)$ is similar to an integer matrix if and only if one of the following cases occurs: 

\smallskip

 Case 1: $at_0, bt_0, ct_0 \in \{2\pi, \pi, \frac{2\pi}{3}, \frac{\pi}{2}, \frac{\pi}{3}\}$.
		
\smallskip

 Case 2: $at_0\in \{2\pi, \pi, \frac{2\pi}{3}, \frac{\pi}{2}, \frac{\pi}{3}\},\; (bt_0,ct_0) \in \{(\frac{2\pi}{5},\frac{4\pi}{5}),(\frac{\pi}{4},\frac{3\pi}{4}),(\frac{\pi}{5},\frac{3\pi}{5}),(\frac{\pi}{6},\frac{5\pi}{6})\}$.

\smallskip
 
 Case 3: $(at_0,bt_0,ct_0)\in \{(\frac{2\pi}{7},\frac{4\pi}{7},\frac{6\pi}{7}),(\frac{2\pi}{9},\frac{4\pi}{9},\frac{8\pi}{9}),(\frac{2\pi}{14},\frac{6\pi}{14},\frac{10\pi}{14}),(\frac{2\pi}{18},\frac{10\pi}{18},\frac{14\pi}{18})\}$.
	
\end{teo}
\begin{proof}
	$\Leftarrow)$ For Cases (1) and (2) the matrices can be conjugated to an integer matrix via a block-matrix (see \cite[Lemma 5.5]{Tol1}). In Case (3), the eigenvalues of $\phi(t_0)$ are all different so $\phi(t_0)$ is similar to the companion matrix of its characteristic polynomial, which is integer.

	$\Rightarrow)$ If any of $at_0, bt_0$ or $ct_0$ are $\pi$ or $2\pi$ the values of the other parameters are the ones obtained for the case $\R\ltimes \R^4$ in \cite[Lemma 5.5]{Tol1}, so we assume next $at_0,bt_0,ct_0\notin \{\pi,2\pi\}$. Now, since the eigenvalues of $\phi(t)$ belong to the unit circle and $\phi(t_0)$ is similar to an integer matrix, it follows from a famous theorem of Kronecker that $\phi(t_0)$ has finite order. Therefore, the characteristic polynomial $P_{\phi(t_0)}$ of $\phi(t_0)$ has degree 6, no real roots and divides $x^d-1$ for some $d\in \N$. Equivalently, $P_{\phi(t_0)}$ is a product of cyclotomic polynomials of degree $\geq 2$. Thus, we are looking for the sets with repetition $S\subset \{3,4,\ldots\}$ which satisfy $\sum_{j\in S} \varphi(j)=6$. The possibilities are $S=\{6\}, S=\{2,4\}$ and $S=\{2,2,2\}$. From there we can deduce the possibilities for $\phi(t_0)$ and looking at the eigenvalues we can deduce the values for $at_0, bt_0, ct_0$ as shown in the statement.
\end{proof}

From the classification of finite subgroups of $\GL(6,\Z)$ we were able to extract the finite cyclic subgroups of $\SL(6,\Z)$, using GAP. We obtained 123 subgroups. Each one of these gives rise to a group $\Z\ltimes_E \Z^{5}$ which is (isomorphic to) a lattice of an almost abelian flat Lie group. Indeed, conjugating $\phi(t_0)$ via matrices in $\GL(6,\R)$ we can obtain each one of the matrices generating these subgroups, due to the following theorem. 

\begin{teo}\cite{Koo}
	A matrix $A\in\GL(k,\R)$ has finite order if and only if $A$ is similar to $\I_{k_1}\oplus (-\I_{k_2})\oplus \matriz{\cos t_1&-\sin t_1\\\sin t_1&\cos t_1}^{d_1}\oplus \cdots \oplus \matriz{\cos t_r&-\sin t_r\\\sin t_r&\cos t_r}^{d_r}$, where $k_1,k_2,r\geq 0$, $d_1,\ldots, d_r\geq 1$, each $t_i$ is a rational multiple of $2\pi$  with $0<t_1<\cdots<t_r<\pi$, and $k_1+k_2+2(d_1+\cdots+d_r)=k$.
\end{teo}

Each of these 123 lattices are non-isomorphic, since we computed the number of subgroups with low index with GAP and this invariant distinguishes them. Thus, we obtain 123 non-diffeomorphic splittable flat solvmanifolds whose holonomy group is finite cyclic.


\subsection{Non almost abelian case} A 7-dimensional non almost abelian flat Lie algebra can be written as $\g=\R^2 \ltimes_{\ad} \R^5$ where $\R^2=\text{span}\{x,y\}$ and in some basis $\B$ of $\z(\g)\oplus [\g,\g]\cong \R^5$, \[\ad_x=(1)\oplus \matriz{0&-a\\a&0}\oplus \matriz{0&-b\\b&0},\quad \ad_y=(1)\oplus \matriz{0&-c\\c&0}\oplus \matriz{0&-d\\d&0},\] where $a^2+c^2\neq 0, b^2+d^2\neq 0$ and $ad-bc\neq 0$.

The corresponding simply-connected Lie group $G$ can be written as $G=\R^2\ltimes_{\phi} \R^5$, where $\phi(tx+sy)=\exp(t\ad_x)\exp(s\ad_y)$. According to Theorem \ref{lattices}, to determine all the splittable lattices in $G$ we have to look for $\{x,y\}$ such that $P^{-1} \exp(\ad_x) P=A$ and $P^{-1}\exp(\ad_y) P=B$ with $A,B\in\GL(5,\Z)$, for some $P\in \GL(5,\R)$.

There are 6079 finite subgroups of $\GL(5,\Z)$. Using GAP, we extract from these the 2-generated abelian finite subgroups of $\SL(5,\Z)$. Some subgroups cannot give rise to a group $\Z^2 \ltimes_{A,B} \Z^5$ isomorphic to a lattice of a flat Lie group since the rank of the abelianization is even. This contradicts the fact that the Kähler even-dimensional flat solvmanifold obtained by multiplying by $S^1$ must have even first Betti number (and $b_1(M\times S^1)=b_1(M)+1$). Discarding these subgroups, we are left with 45 subgroups, which all give rise to a group $\Z^2 \ltimes_{A,B} \Z^5$ which is (isomorphic to) a lattice of a flat Lie group $\R^2\ltimes_\phi \R^5$, as Table 1 shows. Again, we distinguish the lattices computing the number of subgroups of low index. 

Therefore, we get 45 non-diffeomorphic splittable flat solvmanifolds. Note that all these solvmanifolds satisfy the condition $c=-d$ (as Table 1 shows), which will be important in the next section. 

\begin{obs}
	The computations performed in GAP for both the almost abelian case and the non almost abelian case are available in the web page  \href{https://github.com/atolcachier/7-dimensional-splittable-flat-solvmanifolds}{\textcolor{blue}{https://github.com/atolcachier/7-dimensional-splittable-flat-solvmanifolds}}.
\end{obs}
\tiny\begin{center}
	\begin{tabular}{|c|c|c|}
		\hline
		
		$(a,b),(c,d)$ & Matrix which conjugates & Matrices which generate the subgroup \\ 
		
		\hline
		
		$(\pi,2\pi),(\pi,-\pi)$			&	$\matriz{0&0&1\\0&1&0\\1&0&0}\oplus \I_2$	&		$\diag(-1,-1,1,1,1), \diag(-1,-1,1,-1,-1)$ \\
		
		& 					$\matriz{0&0&0&1&0\\ 0&1&0&0&0\\0&0&1&0&0\\1&0&0&0&1\\-1&0&0&0&1}^{-1}$						&       $\diag(1,-1,-1,1,1), (-1)\oplus (-\I_2)\oplus \matriz{0&-1\\-1&0}$ \\
		
		&								$\matriz{0&1&0&0&0\\0&0&0&1&1\\1&0&0&0&0\\0&0&1&0&0\\0&0&0&-1&1}^{-1}$				& 	$(-1)\oplus \I_2 \oplus \matriz{0&-1\\-1&0}, \diag(-1,1,-1,-1,-1)$ \\
		
		& $\matriz{0&0&1&0&1\\0&0&0&1&0\\1&0&0&0&0\\0&0&1&1&-1\\0&1&0&0&0}$	& $(-1)\oplus (1)\oplus \matriz{1&1&0\\0&-1&0\\0&-1&1}, -\I_2\oplus \matriz{0&0&1\\0&-1&0\\1&0&0}$ \\
		
		& 									$\matriz{0&0&1&1&-1\\1&0&0&0&0\\0&0&1&1&1\\0&0&1&-1&1\\0&1&0&0&0}$							& $(-1)\oplus (1)\oplus \matriz{1&0&0\\-1&0&-1\\-1&-1&0}, -\I_2\oplus\matriz{0&1&-1\\0&-1&0\\-1&-1&0}$ \\
		
		& 		$\matriz{0&0&1&0&-1\\0&0&1&0&1\\0&1&0&1&0\\0&1&0&-1&0\\1&0&0&0&0}$			&  $(1)\oplus \matriz{0&0&-1&0\\0&0&0&-1\\-1&0&0&0\\0&-1&0&0}, -\I_2\oplus \matriz{0&0&-1\\0&-1&0\\-1&0&0} $ \\
		
		&   $(1)\oplus \matriz{1&0&0&1\\0&1&1&0\\0&1&-1&0\\1&0&0&-1}$ & $(1)\oplus \matriz{0&0&0&-1\\0&0&-1&0\\0&-1&0&0\\-1&0&0&0}, (1)\oplus -\I_4 $ \\
		
		& 			$\matriz{0&0&1&0&0\\0&1&0&0&1\\1&0&0&1&0\\0&1&1&0&-1\\1&0&0&-1&0}$	&  $\matriz{0&0&0&-1&0\\0&0&0&0&-1\\0&0&1&0&0\\-1&0&0&0&0\\0&-1&0&0&0}, (-1)\oplus\matriz{-1&-1&0&0\\0&1&0&0\\0&0&-1&0\\0&1&0&-1}$ \\
		
		& 		 $\matriz{0&0&1&-1&-1\\1&1&0&0&0\\0&0&1&1&1\\1&-1&0&0&0\\0&0&1&-1&1}$	& $\matriz{0&-1\\-1&0}\oplus \matriz{0&-1&-1\\-1&0&-1\\0&0&1}, -\I_2\oplus \matriz{0&-1&-1\\0&-1&0\\-1&1&0}$\\
		
		\hline
		$(2\pi,\frac{\pi}{2}),(\pi,-\pi)$ & $\matriz{0&0&1&0&0\\0&0&0&1&0\\0&0&0&0&1\\1&0&0&0&0\\0&-1&0&0&0}$     & $\matriz{0&1\\-1&0}\oplus \I_3, -\I_2\oplus (1) \oplus -\I_2$ \\

		&	$\matriz{0&0&1&1&-1\\1&0&0&0&0\\0&1&0&0&0\\0&0&1&0&1\\0&0&0&1&0}$				& $\I_2\oplus \matriz{0&0&-1\\1&0&1\\0&-1&1}, -\I_2\oplus \matriz{0&1&-1\\0&-1&0\\-1&-1&0}$ \\
		
		& $(1)\oplus\matriz{0&1&1&-1\\1&0&0&0\\0&1&0&1\\0&0&1&0}$ & $\I_2\oplus\matriz{0&0&-1\\1&0&1\\0&-1&1}, (1)\oplus (-\I_4)$ \\
		
		& 		$\matriz{0&0&0&1&1\\0&0&0&1&-1\\0&0&1&0&0\\1&0&0&0&0\\0&-1&0&0&0}$	& $\matriz{0&1\\-1&0}\oplus \I_3, -\I_3\oplus \matriz{0&1\\1&0}$ \\
		
		&			$\matriz{0&0&1&0&0\\0&1&1&1&-1\\1&0&0&0&0\\0&1&0&0&1\\0&0&0&-1&0}$			& $ \matriz{1&0&0&0&0\\0&1&0&1&0\\0&0&1&0&0\\0&-1&0&0&-1\\0&-1&0&0&0}, \matriz{-1&0&0&0&0\\0&-1&-1&0&0\\0&0&1&0&0\\0&0&0&-1&0\\0&0&1&0&-1}$ \\
		
		& $\matriz{0&0&0&1&1\\1&0&0&0&0\\0&0&0&1&-1\\0&1&1&0&0\\0&-1&1&-1&-1}$			&  $(1)\oplus \matriz{0&-1&0&0\\1&0&1&1\\0&0&1&0\\0&0&0&1}, (-1)\oplus \matriz{-1&0&-1&-1\\0&-1&1&1\\0&0&0&1\\0&0&1&0}$ \\
			\end{tabular}
	\end{center}
\tiny\begin{center}
\begin{tabular}{|c|c|c|}	

		& $\matriz{0&0&1&0&1\\1&1&0&-1&0\\0&0&1&0&-1\\0&-1&0&0&0\\1&0&0&1&0}$	& $\matriz{0&0&0&-1&0\\1&0&0&1&0\\0&0&1&0&0\\0&-1&0&1&0\\0&0&0&0&1}, -\I_2\oplus \matriz{0&0&1\\0&-1&0\\1&0&0}$ \\
		\hline	
		$(2\pi,\frac{\pi}{2}),(\frac{\pi}{2},-\frac{\pi}{2})$	& $\matriz{0&0&1\\1&0&0\\0&1&0}\oplus \I_2$				& $\I_3\oplus \matriz{0&-1\\1&0}, \matriz{0&-1\\1&0}\oplus (1)\oplus \matriz{0&1\\-1&0}$ \\
		
		&  	$\matriz{0&0&1&1&-1\\0&0&1&0&1\\0&0&0&-1&0\\1&0&0&0&0\\0&1&0&0&0}$	&	$\matriz{0&-1\\1&0}\oplus \I_3, \matriz{0&1\\-1&0}\oplus\matriz{1&1&0\\-1&0&-1\\-1&0&0}$ \\
		
		&		$(1)\oplus\matriz{0&1&0&0\\-1&0&0&-1\\1&1&0&-1\\1&1&2&-1}$		& $(1)\oplus\matriz{0&-1&-1&1\\0&1&0&0\\1&1&1&-1\\1&1&1&0}, (1)\oplus\matriz{0&0&1&-1\\1&0&0&1\\-1&-1&-1&1\\0&-1&-1&1}$ \\
		
		&  $\matriz{0&0&1&0&0\\1&0&0&0&1\\0&-1&0&0&0\\0&0&0&-1&0\\1&1&1&1&-1}$ & $\matriz{0&-1&-1&-1&1\\0&1&0&0&0\\0&0&1&0&0\\1&1&1&1&-1\\1&1&1&1&0}, \matriz{1&1&0&1&0\\-1&0&0&0&-1\\0&0&1&0&0\\-1&-1&-1&-1&1\\-1&0&0&-1&0}$ \\
		
		& $\matriz{0&0&0&1&1\\0&0&0&1&-1\\0&2&0&1&-1\\1&0&1&0&0\\-1&0&1&1&-1}$ & $\matriz{0&0&-1&0&0\\0&1&0&0&0\\1&0&0&-1&1\\0&0&0&1&0\\0&0&0&0&1}, \matriz{0&-1&1&0&0\\0&1&0&1&-1\\-1&1&0&1&-1\\0&-1&0&0&1\\0&1&0&1&0}$ \\
		
		\hline
		
		$(2\pi,\pi), (\frac{\pi}{2},-\frac{\pi}{2})$ & $\matriz{0&0&1\\1&0&0\\0&1&0}\oplus \I_2$ & $\I_3\oplus (-\I_2), \matriz{0&-1\\1&0}\oplus (1)\oplus \matriz{0&1\\-1&0}$ \\
		
		& $\matriz{0&0&1&1&-1\\0&0&1&0&1\\0&0&0&-1&0\\1&0&0&0&0\\0&1&0&0&0}$ & $-\I_2\oplus \I_3, \matriz{0&1\\-1&0}\oplus \matriz{1&1&0\\-1&0&-1\\-1&0&0}$ \\
		
		& $(1)\oplus \matriz{1&0&0&1\\0&1&0&0\\1&1&1&-1\\0&0&1&0}$ & $(1)\oplus \matriz{0&-1&0&1\\0&1&0&0\\0&0&-1&0\\1&1&0&0}, (1)\oplus\matriz{0&0&1&-1\\1&0&0&1\\-1&-1&-1&1\\0&-1&-1&1}$ \\
		
		& $\matriz{0&0&1&0&0\\1&0&0&0&1\\0&-1&0&0&0\\1&1&1&1&-1\\0&0&0&1&0}$ & $\matriz{0&-1&-1&0&1\\0&1&0&0&0\\0&0&1&0&0\\0&0&0&-1&0\\1&1&1&0&0}, \matriz{1&1&0&1&0\\-1&0&0&0&-1\\0&0&1&0&0\\-1&-1&-1&-1&1\\-1&0&0&-1&0}$ \\
		
		& $\matriz{0&0&0&1&1\\0&0&0&1&-1\\0&2&0&1&-1\\1&0&1&0&0\\-1&0&1&1&-1}$ & $\matriz{-1&0&0&1&-1\\0&1&0&0&0\\0&0&-1&-1&1\\0&0&0&1&0\\0&0&0&0&1}, \matriz{0&-1&1&0&0\\0&1&0&1&-1\\-1&1&0&1&-1\\0&-1&0&0&1\\0&1&0&1&0}$ \\
				
		& $(1)\oplus \matriz{1&0&0&1\\0&1&1&0\\1&0&0&-1\\0&1&-1&0}$ & $(1)\oplus\matriz{0&0&0&1\\0&0&1&0\\0&1&0&0\\1&0&0&0}, (1)\oplus \matriz{0&0&-1&0\\0&0&0&1\\1&0&0&0\\0&-1&0&0}$\\
		
		& $\matriz{0&0&0&0&1\\1&0&1&0&0\\0&-1&0&-1&0\\1&0&-1&0&1\\0&-1&0&1&-1}$ & $\matriz{0&0&1&0&-1\\0&0&0&1&-1\\1&0&0&0&1\\0&1&0&0&1\\0&0&0&0&1}, \matriz{0&0&0&1&-1\\0&0&-1&0&0\\0&1&0&0&1\\-1&0&0&0&0\\0&0&0&0&1}$ \\
		
		& $\matriz{1&-1&-1&1&-1\\1&0&1&1&0\\0&1&0&0&1\\1&0&1&-1&0\\0&-1&0&0&1}$ & $\matriz{0&0&-1&1&0\\0&0&0&0&1\\0&0&1&0&0\\1&0&1&0&0\\0&1&0&0&0}, \matriz{1&-1&0&1&0\\1&0&1&0&0\\-1&0&0&-1&0\\0&0&0&0&-1\\0&0&0&1&0}$ \\
		
		\hline 
		$(\frac{\pi}{3},2\pi),(\pi,-\pi)$ & $\matriz{0&0&1\\0&-\sqrt{3}&0\\2&-1&0}\oplus \I_2$ & $\matriz{1&-1\\1&0}\oplus \I_3, -\I_2\oplus (1)\oplus (-\I_2)$ \\
		
			& $\matriz{0&0&0&1&1\\0&-1&2&0&0\\0&\sqrt{3}&0&0&0\\0&0&0&1&-1\\1&0&0&0&0}$ & $(1)\oplus\matriz{0&1\\-1&1}\oplus \I_2, -\I_3\oplus \matriz{0&1\\1&0}$ \\
		\end{tabular}
	\end{center}

\tiny\begin{center}
	\begin{table}[H]
\begin{tabular}{|c|c|c|}

			$(\frac{2\pi}{3},\frac{\pi}{2}),(\pi,-\pi)$ & $\matriz{0&0&1&0&0\\2&-1&0&0&0\\0&\sqrt{3}&0&0&0\\0&0&0&1&0\\0&0&0&0&-1}$ & $\matriz{0&-1\\1&-1}\oplus (1)\oplus \matriz{0&1\\-1&0}, -\I_2\oplus (1)\oplus (-\I_2)$ \\
			
			& $\matriz{0&0&1&0&1\\2&-1&0&0&0\\0&\sqrt{3}&0&0&0\\0&0&1&1&-1\\0&0&0&-1&0}$ &$\matriz{0&-1\\1&-1}\oplus \matriz{1&1&0\\-1&-1&1\\0&-1&1}, -\I_2\oplus \matriz{0&0&1\\0&-1&0\\1&0&0}$ \\
			
			\hline 
			
			$(\frac{\pi}{3},\frac{2\pi}{3}),(\pi,-\pi)$ & $\matriz{0&0&1&0&0\\2&-1&0&0&0\\0&\sqrt{3}&0&0&0\\0&0&0&-1&2\\0&0&0&\sqrt{3}&0}$ & $\matriz{1&-1\\1&0}\oplus (1)\oplus \matriz{-1&1\\-1&0}, -\I_2\oplus (1)\oplus (-\I_2)$ \\
			
			& $(1)\oplus \matriz{0&\sqrt{3}&\sqrt{3}&0\\2&-1&-1&2\\2&-1&1&-2\\0&\sqrt{3}&-\sqrt{3}&0}$ &$(1)\oplus \matriz{0&0&1&0\\0&0&1&-1\\-1&1&0&0\\0&1&0&0}, (1)\oplus(-\I_4)$ \\
			
			\hline
			
			$(2\pi,\frac{\pi}{3}),(\frac{2\pi}{3},-\frac{2\pi}{3})$ & $\matriz{0&0&1&0&0\\2&1&0&0&0\\0&\sqrt{3}&0&0&0\\0&0&0&-1&2\\0&0&0&\sqrt{3}&0}$ & $\I_3\oplus \matriz{0&1\\-1&1}, \matriz{-1&-1\\1&0}\oplus(1)\oplus\matriz{0&-1\\1&-1}$ \\
			
			& $\matriz{0&0&1&1&-1\\0&0&-1&2&1\\0&0&\sqrt{3}&0&\sqrt{3}\\-\sqrt{3}&0&0&0&0\\-1&2&0&0&0}$ & $\matriz{0&1\\-1&1}\oplus \I_3, \matriz{0&-1\\1&-1}\oplus\matriz{0&1&0\\0&0&-1\\-1&0&0}$ \\
			
			\hline 
			
			$(2\pi,\frac{2\pi}{3}),(\frac{\pi}{3},-\frac{\pi}{3})$ & $\matriz{0&0&1&0&0\\2&-1&0&0&0\\0&\sqrt{3}&0&0&0\\0&0&0&2&1\\0&0&0&0&\sqrt{3}}$ & $\I_3\oplus\matriz{-1&-1\\1&0},\matriz{1&-1\\1&0}\oplus(1)\oplus\matriz{1&1\\-1&0}$ \\
			
			& $(1)\oplus\matriz{2&-2&2&1\\0&0&0&\sqrt{3}\\-1&1&2&-2\\\sqrt{3}&\sqrt{3}&0&0}$ & $(1)\oplus\matriz{0&0&1&-1\\-1&0&0&0\\0&-1&0&1\\0&0&0&1},(1)\oplus\matriz{0&0&-1&0\\1&0&0&1\\1&0&1&0\\1&-1&1&1}$ \\
			\hline
			$(2\pi,\frac{\pi}{3}),(\frac{\pi}{3},-\frac{\pi}{3})$ & $\matriz{0&0&1&0&0\\2&-1&0&0&0\\0&\sqrt{3}&0&0&0\\0&0&0&-1&2\\0&0&0&\sqrt{3}&0}$ & $\I_3\oplus\matriz{0&1\\-1&1}, \matriz{1&-1\\1&0}\oplus(1)\oplus\matriz{1&-1\\1&0}$ \\

			$(2\pi,\frac{2\pi}{3}),(\frac{2\pi}{3},-\frac{2\pi}{3})$ & $\matriz{0&0&1&0&0\\2&-1&0&0&0\\0&\sqrt{3}&0&0&0\\0&0&0&-1&2\\0&0&0&\sqrt{3}&0}$ & $\I_3\oplus\matriz{-1&1\\-1&0}, \matriz{0&-1\\1&-1}\oplus(1)\oplus\matriz{0&-1\\1&-1}$ \\
			
			& $\matriz{0&0&1&-1&-1\\0&0&2&1&1\\0&0&0&-\sqrt{3}&\sqrt{3}\\-1&2&0&0&0\\\sqrt{3}&0&0&0&0}$ & $\matriz{-1&1\\-1&0}\oplus \I_3, \matriz{0&-1\\1&-1}\oplus\matriz{0&0&-1\\-1&0&0\\0&1&0}$ \\
			
			&  $\matriz{1&0&0&0&0\\0&2&1&1&-1\\0&0&\sqrt{3}&\sqrt{3}&\sqrt{3}\\0&2&-2&1&2\\0&0&0&-\sqrt{3}&0}$ & $\matriz{1&0&0&0&0\\0&1&0&1&0\\0&1&0&1&1\\0&-1&1&-1&-1\\0&0&0&1&1},\matriz{1&0&0&0&0\\0&-1&0&-1&-1\\0&0&0&0&-1\\0&1&-1&0&1\\0&0&1&0&-1}$ \\
			
			& $\matriz{1&-1&1&-1&-1\\0&2&1&0&-1\\0&0&\sqrt{3}&0&\sqrt{3}\\1&0&0&2&0\\\sqrt{3}&0&0&0&0}$ & $\matriz{0&0&0&1&0\\0&1&0&1&0\\0&0&1&-1&0\\-1&0&0&-1&0\\0&0&0&1&1}, \matriz{-1&0&0&-1&0\\-1&0&-1&0&0\\1&0&0&0&-1\\1&0&0&0&0\\-1&1&0&0&0}$ \\
			
			& $\matriz{1&1&-1&-1&1\\2&-1&-2&1&-1\\0&\sqrt{3}&0&\sqrt{3}&\sqrt{3}\\2&2&1&1&-1\\0&0&-\sqrt{3}&-\sqrt{3}&-\sqrt{3}}$ & $\matriz{0&-1&-1&-1&0\\0&1&1&1&1\\-1&-1&0&-1&0\\0&0&0&1&0\\0&0&-1&-1&0}, \matriz{0&1&0&-1&0\\0&-1&-1&0&0\\0&1&0&0&0\\0&-1&0&0&-1\\1&1&0&0&0}$ \\
			\hline
		\end{tabular}
		\caption{7-dimensional splittable non almost abelian flat solvmanifolds}
		\label{Table 2}
	\end{table}
\end{center}

\normalsize

\section{Examples of $G_2$-structures on flat solvmanifolds}
The aim of this section is to study the existence of invariant closed and coclosed $G_2$-structures in the flat solvmanifolds we found in the previous section.

\subsection{Almost abelian solvmanifolds}
Let $\g_{a,b,c}=\R x\ltimes_{\ad_x} \R^6$ be a flat almost abelian Lie algebra and $G_{a,b,c}=\R \ltimes_{\phi} \R^6$ its corresponding simply-connected Lie group, where \begin{align*} 
	\ad_x&=\matriz{0&-a\\a&0}\oplus\matriz{0&-b\\b&0}\oplus\matriz{0&-c\\c&0},\quad a^2+b^2+c^2\neq 0,\\
\phi(t)&=\matriz{\cos(at)&-\sin(at)\\ \sin(at)& \cos (at)} \oplus \matriz{\cos(bt)&-\sin(bt)\\ \sin(bt)&\cos(bt)}\oplus\matriz{\cos(ct)&-\sin(ct)\\ \sin(ct)&\cos(ct)}.
\end{align*}
The Lie brackets of $\g_{a,b,c}$ are given by 
\begin{align*}
	[e_1,e_2]= a e_3,& \quad [e_1, e_4]=b e_5,\quad [e_1,e_6]=c e_7, \\
	[e_1,e_3]= -a e_2&, \quad [e_1, e_5]=-b e_4, \quad [e_1,e_7]= -c e_6.
\end{align*}
Therefore, the Chevalley-Eilenberg differential $\d\,:\alt^1 \g_{a,b,c}^*\to \alt^2\g_{a,b,c}^*$ is given by 
\begin{align}\label{dif1}
	\d e^2= a e^{13}, &\quad \d e^3=-a e^{12}, \quad \d e^4=b e^{15},\\
	\d e^5=-b e^{14}, &\quad \d e^6=c e^{17}, \quad \d e^7=-c e^{16}. \nonumber
\end{align}

Let $\varphi\in \alt^3 \g_{a,b,c}^*$ be the positive form given by  \begin{equation}\label{phi} \varphi=e^{123}+e^{145}+e^{167}+e^{246}-e^{257}-e^{347}-e^{356}.
	\end{equation}  Note that $\{e_1,\ldots,e_7\}$ is an orthonormal basis for the induced metric $g_\varphi$.
 \begin{prop}
 	The 3-form $\varphi$ as above is coclosed for any choice of $a,b,c$, and it is closed (therefore torsion-free) if and only if $a+b+c=0$.
 \end{prop}
\begin{proof}
	We compute $\d \varphi$ using \eqref{dif1}: \[\d\varphi=(a+b+c)(e^{2467}-e^{2357}-e^{1457}-e^{1367}).\] 
	In addition, \[\star_\varphi\varphi=-e^{1247}-e^{1256}-e^{1346}+e^{1357}+e^{2345}+e^{2367}+e^{4567}.\] Again, using \eqref{dif1} it is easily obtained that $\d\star_\varphi\varphi=0$.
\end{proof}
\begin{obs}
	This proposition coincides with \cite{F1,F2} where the existence of closed and coclosed $G_2$-structures on any almost abelian Lie algebra is studied. Indeed, there it is established that $\varphi$ is closed if and only if $\ad_x\in \mathfrak{sl}(3,\C)$ (i.e. $a+b+c=0$) and is coclosed if and only if $\ad_x\in \mathfrak{sp}(3,\R)$ (which always happens in our case, because $\ad_x$ is skew-symmetric).
\end{obs}

\medskip

\begin{prop}
	Up to isomorphism of the induced lattices, the values of $(at_0, bt_0, ct_0)$ such that $\phi(t_0)$ is similar to an integer matrix and  $at_0+bt_0+ct_0=0$ are the following:
	\[\left(\frac{2\pi}{7},\frac{4\pi}{7},-\frac{6\pi}{7}\right), \left(\pi,-\frac{\pi}{6},-\frac{5\pi}{6}\right),\left(\frac{2\pi}{3},\frac{\pi}{6},-\frac{5\pi}{6}\right),\left(\pi,-\frac{\pi}{4},-\frac{3\pi}{4}\right),\left(\frac{\pi}{2},-\frac{\pi}{4},\frac{3\pi}{4}\right),\]
	\[\left(2\pi,2\pi,-4\pi\right),\left(2\pi,-\pi,-\pi\right),\left(2\pi,-\frac{\pi}{2},-\frac{3\pi}{2}\right),\left(2\pi, -\frac{\pi}{3},-\frac{5\pi}{3}\right),\left(2\pi,-\frac{2\pi}{3},-\frac{4\pi}{3}\right),\]
	\[\left(\pi,-\frac{\pi}{2},-\frac{\pi}{2}\right),\left(\pi,-\frac{2\pi}{3},-\frac{\pi}{3}\right),\left(\frac{\pi}{3},\frac{\pi}{3},-\frac{2\pi}{3}\right),\left(\frac{2\pi}{3},\frac{2\pi}{3},-\frac{4\pi}{3}\right).\]
\end{prop}
\begin{proof} 
The characteristic polynomial $P_{\phi(t_0)}$ is given by \[P_{\phi(t_0)}=(x^2-2x\cos(at_0)+1)(x^2-2x\cos(bt_0)+1)(x^2-2x\cos(ct_0)+1).\] The values of $at_0, bt_0, ct_0$ such that $\phi(t_0)$ is similar to an integer matrix were obtained in Theorem \ref{valuest0}. We can change the values of $at_0, bt_0, ct_0$ by $\{\pm at_0\}+2\pi \Z$, $\{\pm bt_0 \}+2\pi \Z$ and $\{\pm ct_0\}+2\pi \Z$ so that we do not change the value of the respective cosines.  Moreover, we want $at_0+bt_0+ct_0=0$, so we have to verify, for the values obtained in Theorem \ref{valuest0} if \[0\in (\{\pm at_0\}+2\pi \Z)+(\{\pm bt_0 \}+2\pi \Z)+(\{\pm ct_0 \}+2\pi \Z).\] Equivalently, we have to check if some sum
$\{\pm a t_0\}+\{\pm bt_0\}+\{\pm ct_0\}$ is equal to $2\pi k$, for some $k\in \Z$. In fact, we have to check only if $\{\pm at_0\}+\{\pm bt_0\}+ ct_0=2\pi k$. This can be done by a straightforward computation and thus the values of the statement are obtained. \end{proof}

With the values of $at_0, bt_0, ct_0$ obtained, we list in the following table to which integer matrices (up to integral similarity) we can conjugate $\phi(t_0)$. For each one of these triples we obtain non-isomorphic lattices (as we saw before) and  therefore, we get 30 splittable flat solvmanifolds with a torsion-free $G_2$-structure. All of these examples have finite cyclic holonomy contained in $G_2$ which can be computed easily from Proposition  \ref{hol}.

\medskip

\tiny\begin{center} 
\begin{tabular}{|c|l|c|}
	\hline
	$(at_0,bt_0,ct_0)$ & Similar to & $\hol(\Gamma\backslash G_{a,b,c})$ \\
	
	\hline
	
	$(\frac{2\pi}{7},\frac{4\pi}{7},-\frac{6\pi}{7})$ & $\matriz{-1&0&0&0&1&0\\-1&0&1&0&0&0\\-1&0&0&0&0&0\\1&-1&0&0&0&0\\-1&0&0&0&0&-1\\1&0&0&1&0&0}$& $\Z_7$ \\
	
	\hline
	
	$(\pi,-\frac{\pi}{6},-\frac{5\pi}{6})$ & $\matriz{-1&-1&0&0\\1&1&1&-1\\1&0&1&0\\1&0&1&-1}\oplus (-\I_2)$ & $\Z_{12}$ \\
	
	\hline
	
	$(\frac{2\pi}{3},\frac{\pi}{6},-\frac{5\pi}{6})$ & $\matriz{0&0&0&1\\0&0&1&-1\\1&0&0&0\\0&-1&0&0}\oplus \matriz{0&-1\\1&-1}$, $\matriz{0&0&0&0&0&-1\\0&0&1&0&0&-1\\-1&0&0&0&1&0\\0&-1&0&0&0&1\\0&0&0&1&0&0\\0&0&0&0&1&-1}$ & $\Z_{12}$\\
	
	\hline 
	
	$(\pi,-\frac{\pi}{4},-\frac{3\pi}{4})$ & $\matriz{0&1&0&0\\0&0&-1&0\\0&0&0&1\\1&0&0&0}\oplus (-\I_2)$, $\matriz{-1&0&0&0&0&0\\0&0&-1&-1&0&1\\0&0&0&1&0&0\\0&0&0&0&-1&0\\0&-1&0&0&0&-1\\0&0&0&1&0&-1}$ & $\Z_8$ \\
	
	\hline 
	
	$(\frac{\pi}{2},-\frac{\pi}{4},\frac{3\pi}{4})$ & $\matriz{0&0&0&-1\\0&0&1&0\\-1&0&0&0\\0&-1&0&0}\oplus \matriz{0&-1\\1&0}$, $\matriz{-1&0&-1&0&-1&0\\0&0&1&0&0&0\\1&0&0&0&0&1\\0&-1&0&0&0&0\\1&1&1&1&1&-1\\1&0&1&1&1&0}$& $\Z_8$\\ &$\matriz{0&0&1&1&0&0\\0&0&-1&0&0&-1\\0&-1&0&0&1&0\\0&1&0&0&0&0\\0&0&-1&0&0&0\\1&1&0&0&0&0}$ & \\
	
	\hline
	
	$(2\pi,2\pi,-4\pi)$ & $\I_6$  & $\{e\}$\\
	
	\hline 
	
	$(2\pi,-\pi,-\pi)$ & $-\I_4\oplus \I_2$, $(-\I_3)\oplus(1)\oplus \matriz{0&-1\\-1&0}$, $-\I_2\oplus\matriz{0&0&-1&0\\0&0&0&-1\\-1&0&0&0\\0&-1&0&0}$ & $\Z_2$ \\
	
\end{tabular}
\end{center}
\tiny\begin{center} 
\begin{table}
\begin{tabular}{|c|l|c|}
	
	$(2\pi,-\frac{\pi}{2},-\frac{3\pi}{2})$ & $\matriz{0&-1&0&-1\\1&0&1&0\\0&0&0&1\\0&0&-1&0}\oplus \I_2$, $\matriz{1&0&0\\0&0&1\\0&-1&0}\oplus\matriz{1&0&0\\-1&0&1\\0&-1&0}$,&$\Z_4$ \\
	&$(1)\oplus\matriz{0&1\\-1&0}\oplus\matriz{1&0&0\\-1&0&1\\0&-1&0}$  ,
	$\matriz{0&0&0&0&-1&0\\0&0&0&0&0&-1\\0&1&0&0&0&1\\-1&0&0&0&-1&0\\0&0&0&1&1&0\\0&0&-1&0&0&1}$ & \\
	
	\hline
	
	$(2\pi, -\frac{\pi}{3},-\frac{5\pi}{3})$ & $\matriz{1&1\\-1&0}\oplus\matriz{0&-1\\1&1}\oplus \I_2$ & $\Z_6$ \\
	
	\hline
	
	$(2\pi,-\frac{2\pi}{3},-\frac{4\pi}{3})$ & $\matriz{-1&1\\-1&0}\oplus\matriz{-1&-1\\0&1}\oplus \I_2$, $(1)\oplus \matriz{-1&-1\\1&0}\oplus \matriz{0&0&1\\-1&0&0\\0&-1&0}$,& $\Z_3$\\& $\matriz{0&-1&0&0&0&0\\0&0&0&0&1&0\\0&0&0&1&0&0\\0&0&0&0&0&-1\\-1&0&0&0&0&0\\0&0&-1&0&0&0}$ & \\
	
	\hline
	
	$(\pi,-\frac{\pi}{2},-\frac{\pi}{2})$ & $\matriz{0&1&0&1\\-1&0&-1&0\\0&0&0&-1\\0&0&1&0}\oplus(-\I_2)$, $(-1)\oplus\matriz{0&-1\\1&0}\oplus\matriz{-1&0&0\\1&0&-1\\0&1&0}$,& $\Z_4$\\& $\matriz{0&0&0&0&1&0\\0&0&0&0&0&1\\0&-1&0&0&0&-1\\1&0&0&0&1&0\\0&0&0&-1&-1&0\\0&0&1&0&0&-1}$ &  \\
	
	$(\pi,-\frac{2\pi}{3},-\frac{\pi}{3})$ & $\matriz{1&-1\\1&0}\oplus \matriz{-1&-1\\1&0}\oplus (-\I_2)$, $(-\I_2)\oplus\matriz{0&-1&1&0\\1&0&0&0\\1&0&0&-1\\0&0&1&0}$,&$\Z_6$\\& $(-1)\oplus \matriz{-1&-1\\1&0}\oplus\matriz{0&0&-1\\1&0&0\\0&1&0}$,   $(-1)\oplus\matriz{0&-1&0&0&-1\\1&-1&0&0&0\\0&-1&0&0&0\\0&1&-1&0&0\\0&1&0&1&0}$&\\
	
	\hline
	
	$(\frac{\pi}{3},\frac{\pi}{3},-\frac{2\pi}{3})$ & $\matriz{1&1\\-1&0}\oplus\matriz{1&-1\\1&0}\oplus\matriz{0&-1\\1&-1}$, $\matriz{1&-1\\1&0}\oplus\matriz{0&0&1&0\\1&0&0&0\\0&0&0&1\\0&-1&-1&0}$ & $\Z_6$\\
	\hline
	$(\frac{2\pi}{3},\frac{2\pi}{3},-\frac{4\pi}{3})$ & $\matriz{0&0&0&0&0&1\\0&0&-1&0&0&0\\0&1&-1&0&0&0\\0&0&0&0&-1&0\\0&0&0&1&-1&0\\-1&0&0&0&0&-1}$ &$\Z_3$ \\
	\hline
\end{tabular}
\caption{Almost abelian 7-dimensional flat solvmanifolds with a torsion-free $G_2$-structure}
\end{table}
\end{center}
 
\normalsize
 
\subsection{Non-almost abelian solvmanifolds}

Let $\g_{a,b,c,d}=\R^2\ltimes_{\ad} \R^5$ be a flat non almost abelian Lie algebra and $G_{a,b,c,d}=\R^2 \ltimes_{\phi} \R^5$ its corresponding simply-connected Lie group, where $\R^2=\text{span}_\R\{x,y\}$ and \begin{align*} 
	\ad_x=(1)\oplus \matriz{0&-a\\a&0}\oplus \matriz{0&-b\\b&0},\quad \ad_y=(1)\oplus \matriz{0&-c\\c&0}\oplus \matriz{0&-d\\d&0},
\end{align*} where, $a^2+c^2\neq 0\neq b^2+d^2, ad-bc\neq 0$.

The Lie brackets are given by
\begin{align*}
	[e_1, e_4]=a e_5&, \quad [e_1, e_6]=b e_7, \quad [e_2, e_4]=c e_5, \quad [e_2, e_6]=d e_7,\\
	[e_1, e_5]=-a e_4&, \quad [e_1, e_7]=-b e_6, \quad [e_2, e_5]=-c e_4, \quad [e_2, e_7]=-d e_6.
\end{align*}
Therefore, the Chevalley-Eilenberg differential $\d\,:\alt^1 \g_{a,b,c,d}\to\alt^2\g_{a,b,c,d}$ is given by 
\begin{align}\label{dif2}
	\d e^4=ae^{15}+ce^{25},&\quad \d e^5=-ae^{14}-ce^{24},\\
	\d e^6=be^{17}+de^{27},&\quad \d e^7=-be^{16}-de^{26}. \nonumber
\end{align}

We want to study the existence of closed and coclosed $G_2$-structures in $\g_{a,b,c,d}$. We will prove that $\g_{a,b,c,d}$ does not admit any closed $G_2$-structure. The key lemma is the following one, proved in \cite{Fino}, where the following notation is used. Given a 7-dimensional real Lie algebra $\g$, every 3-form $\phi\in \alt^3\g^*$ on $\g$ gives rise to a symmetric bilinear map $b_\phi$ by setting $b_\phi:\g\times\g\to\alt^7\g^*\simeq \R$, \[(v,w)\mapsto \frac{1}{6} \iota_v \phi \wedge \iota_w \phi \wedge \phi.\]
\begin{lema}\cite{Fino}\label{Fino}
	A 7-dimensional oriented real Lie algebra $\g$ does not admit any closed $G_2$-structure if for every closed 3-form $\phi\in\alt^3 \g^*$ one of the following conditions hold for the map $b_\phi: \g\times\g\to\alt^7 \g^*\simeq \R:$
	\begin{enumerate}
		\item There exists $v\in\g\setminus\{0\}$ such that $b_\phi(v,v)=0$,
		
		\item There exist $v,w\in\g\setminus\{0\}$ such that $b_\phi(v,v)b_\phi(w,w)\leq 0$.
	\end{enumerate}
\end{lema}

\begin{prop}
	The Lie algebra $\g_{a,b,c,d}$ does not admit closed $G_2$-structures.
\end{prop}
\begin{proof}
	Let $\phi=\sum_{i<j<k} a_{ijk} e^{ijk}$ be a generic $3$-form. For $\phi$ to be closed we have
	\begin{align*} 
	&0=-(a a_{235}-c a_{135}) e^{1234}+(a a_{234}-c a_{134}) e^{1235} \\
&+(d a_{137}-b a_{237}) e^{1236}-(d a_{136}-b a_{236}) e^{1237} \\
&	 - (a a_{256}-d a_{147}-c a_{156}+b a_{247}) e^{1246}- (a a_{257}+d a_{146}-c a_{157}-b a_{246})e^{1247}\\
	&+ (a a_{246}-c a_{146}+d a_{157}-b a_{257}) e^{1256}
	+(a a_{247}-c a_{147}-d a_{156}+b a_{256}) e^{1257}\\
	&-(a a_{356}+b a_{347}) e^{1346}
	-(a a_{357}-b a_{346})e^{1347}+(a a_{346}-b a_{357}) e^{1356}
	+(a a_{347}+b a_{356}) e^{1357}\\&-b a_{457} e^{1456}+ b a_{456} e^{1457}
	- a a_{567} e^{1467}+ a a_{467} e^{1567}- (d a_{347}+c a_{356}) e^{2346}\\&
	+ (d a_{346}-c a_{357})e^{2347}+(c a_{346}-d a_{357}) e^{2356}
	+ (c a_{347}+d a_{356})e^{2357}\\&- d a_{457} e^{2456}+ d a_{456} e^{2457}- c a_{567} e^{2467}
	+ c a_{467} e^{2567}.
	\end{align*} 
Since $a^2+c^2\neq 0$ and $b^2+d^2\neq 0$, we have $a_{467}=a_{567}=0$ and $a_{456}=a_{457}=0$ respectively. Looking at the last eight pairs of terms we deduce that
\[ \begin{cases}
	a a_{347}=-b a_{356} \\
	a a_{356}=-b a_{347} \\
	a a_{357}=b a_{346} \\
	a a_{346}= b a_{357}
\end{cases}\quad\text{and}\quad \begin{cases}
d a_{347}=-c a_{356}\\
d a_{346}= c a_{357}\\
c a_{346}= d a_{357}\\
c a_{347}= -d a_{356}
\end{cases}. \]
From here we have
\[ a^2 a_{346}=ab a_{357}= b^2 a_{346} \]
\[ a^2 a_{347}= -ab a_{356} = b^2 a_{347}  \]
\[ c^2 a_{346}= cd a_{357}= d^2 a_{346}\]
\[ c^2 a_{347}= -cd a_{356}= d^2 a_{347}\]
If $a_{346}\neq 0\neq a_{347}$ then $a^2=b^2$ and $c^2=d^2$. Thus, $b=\pm a$ and $c=\pm d$, neither of them being zero. The condition $ad-bc\neq 0$ rules out the cases $b=a, c=d$ and $b=-a, c=-d$. In the other two cases, looking at the previous equations it follows that $a=b=c=d=0$, which contradicts $ad-bc\neq 0$. Therefore, $a_{346}=a_{347}=a_{356}=a_{357}=0$.

Furthermore, from the terms having 4 summands we see that 
\[a_{157}= \frac{(b^2-a^2) a_{246}+(ac-bd)a_{146} }{ad-bc},\quad a_{257}=\frac{(-ac+bd)a_{246}+(c^2-d^2)a_{146}}{ad-bc}, \]
\[a_{156}=\frac{(-ac+bd)a_{147}+(a^2-b^2)a_{247}}{ad-bc}, \quad a_{256}=\frac{(ac-bd)a_{247}+(d^2-c^2)a_{147}}{ad-bc}.\]

Using the values we have just found in the expresion for $\varphi$ we compute now \[\iota_{e_4} \phi \wedge \iota_{e_4} \phi \wedge \phi =-6(a_{146}a_{345}a_{247}-a_{147}a_{345}a_{246}) e^{1\ldots 7},\] 
\[\iota_{e_5}\phi \wedge \iota_{e_5}\phi \wedge \phi=6(a_{146}a_{247}a_{345}-a_{147}a_{246}a_{345})e^{1\ldots 7}.\]
Using (ii) of Lemma \ref{Fino} for $v=e_4$ and $w=e_5$, we conclude that $\g_{a,b,c,d}$ does not admit any closed $G_2$-structure, for any choice of values $a,b,c,d$.
\end{proof}

\medskip

Although $\g_{a,b,c,d}$ does not admit closed $G_2$-structures, it does admit coclosed $G_2$-structures for some values of $a,b,c,d$.

\begin{prop}\label{coclosed}
Let $\varphi\in \alt^3 \g_{a,b,c,d}^*$ be given as in \eqref{phi}. Then $\varphi$ is coclosed if and only if $c=-d$.
\end{prop}
\begin{proof}
	Using \eqref{dif2} we compute 
	\[\d\star_\varphi\varphi=(d+c) e^{12347}+ (d+c) e^{12356}.\]
\end{proof}

 Since the 45 non-almost abelian 7-dimensional splittable flat solvmanifolds appearing in Table 1 satisfy $d=-c$, all these solvmanifolds admit a coclosed $G_2$-structure.

\medskip

\subsection{Divergence-free examples}
Finally, we will show that all 45 non almost abelian 7-dimensional flat solvmanifolds we have obtained admit a divergence-free $G_2$-structure. First, we give the formulas of the torsion forms $\tau_1,\ldots,\tau_4$ for the Lie algebra $\g_{a,b,c,d}$. 
\begin{prop}\label{torsionforms}
	Let $\varphi\in \Lambda^3 \g_{a,b,c,d}^*$ defined by $\varphi=e^{123}+e^{145}+e^{167}+e^{246}-e^{257}-e^{347}-e^{356}$. Then, the torsion forms of $\varphi$ are given by
	\begin{align*} 
		\tau_0&=-\frac{4}{7}(b+a),\\
		\tau_1&=-\frac{1}{6}(d+c) e^3,\\
		\tau_2&=-(d+c)(e^{47}+e^{56}),\\
		\tau_3&=\frac{1}{7}(b+a)(3e^{257}-3e^{246}+3e^{347}+3e^{356}+4e^{123}+4e^{145}+4e^{167})\\&\quad+\frac{1}{2}(d+c)(e^{146}-e^{157}+e^{245}+e^{267}).
	\end{align*}
\end{prop}
\begin{proof}
	We compute $\tau_0,\ldots,\tau_3$ using equations \eqref{tau3}. 

	Since $\tau_0=\frac{1}{7} \star_\varphi (\d\varphi \wedge \varphi)$, we compute 
	\[ \d \varphi=(b+a)(e^{1247}+e^{1256}+e^{1346}-e^{1357})+(d-c)(e^{2346}-e^{2357}),\quad  \d \varphi\wedge \varphi= -4(b+a) e^{1\ldots 7}.\]
	Therefore, \[\tau_0=-\frac{4}{7}(b+a).\]
	
	Now, for $\tau_1=-\frac{1}{12} \star_\varphi ( \star_\varphi  \d\varphi \wedge \varphi)$, we compute 
	\[ \star_\varphi\d\varphi=(b+a)(e^{257}-e^{246}+e^{347}+e^{356})+(d+c) (e^{146}-e^{157}),\quad\star_\varphi\d\varphi \wedge \varphi=2(d+c) e^{124567}.\]
	Therefore, \[\tau_1=-\frac{1}{6}(d+c) e^3.\]
	
	We know that $\tau_2=- \star_\varphi \d\star_\varphi \varphi+ 4 \star_\varphi (\tau_1 \wedge \d\star_\varphi \varphi)$. It follows from Proposition \ref{coclosed} and the expression of $\tau_1$ that 
	\[ \tau_1 \wedge \d\star_\varphi\varphi=0.\]
	Therefore, \[ \tau_2=-\star_\varphi\d\star_\varphi\varphi=-(d+c)(e^{47}+e^{56}).\]
	
	Finally, for $\tau_3= \star_\varphi \d\varphi-\tau_0 \varphi-3 \star_\varphi (\tau_1\wedge \varphi)$, we compute 
	\[\tau_1 \wedge \varphi=\frac{1}{6}(d+c)(e^{1345}+e^{1367}+e^{2346}-e^{2357}),\quad\star_\varphi(\tau_1 \wedge \varphi)=\frac{1}{6}(d+c)(e^{146}-e^{157}-e^{245}-e^{267}).\] Therefore,
	\begin{align*}
		\tau_3&=(b+a)(e^{257}-e^{246}+e^{347}+e^{356})+(d+c)(e^{146}-e^{157})\\
		&\quad +\frac{4}{7}(b+a)(e^{123}+e^{145}+e^{167}+e^{246}-e^{257}-e^{347}-e^{356})  \\
		&\quad -\frac{1}{2}(d+c)(e^{146}-e^{157}-e^{245}-e^{267})\\
		&=\frac{1}{7}(b+a)(3e^{257}-3e^{246}+3e^{347}+3e^{356}+4e^{123}+4e^{145}+4e^{167})\\&\quad+\frac{1}{2}(d+c)(e^{146}-e^{157}+e^{245}+e^{267})
	\end{align*}
\end{proof}
Next, we recall the \textit{divergence} of $T_\varphi$. It is defined as the vector field $\operatorname{div} T_\varphi$ given by
\begin{equation}\label{div}
	g_\varphi(\div T_\varphi, E_j)=\sum_{i=1}^7 (\nabla_{E_i} T_\varphi) (E_i,E_j),
\end{equation} where $\{E_i\}_{i=1}^7$ is an orthonormal local frame respect to the induced metric $g_\varphi$.

	In the Lie algebra setting, given that the basis $\{e_1,\ldots,e_7\}$ of $\g_{a,b,c,d}$ is an orthonormal basis for $\la \cdot,\cdot\ra_\varphi$, equation \eqref{div} takes the following form: \begin{equation}\label{div2} 
	\la \div(T_{\varphi}),e_j\ra_{\varphi}=-\sum_{i=1}^7  T_{\varphi} (\nabla_{e_i} e_i, e_j)- \sum_{i=1}^7 T_{\varphi}(e_i, \nabla_{e_i} e_j),
\end{equation} 

\begin{teo}
		Let $\varphi\in \Lambda^3 \g_{a,b,c,d}^*$ defined by $\varphi=e^{123}+e^{145}+e^{167}+e^{246}-e^{257}-e^{347}-e^{356}$. Then, for any choice of values $(a,b,c,d)$ we have $\operatorname{div} T_\varphi=0$, i.e., $\varphi$ is divergence-free.
\end{teo}
\begin{proof} 
We compute $\nabla_{e_i} e_i$ and $\nabla_{e_i} e_j$ using Koszul's formula. \[ 2\la \nabla_{e_i} e_j, e_k\ra_\varphi =\la [e_i,e_j],e_k\ra_\varphi-\la [e_j,e_k],e_i\ra_\varphi+\la [e_k,e_i],e_j\ra_\varphi\quad \forall i,j,k.\]

We obtain that $\nabla_{e_i} e_i=0$ for all $1\leq i\leq 7$ and 
\begin{align*} 
	\nabla_{e_1} e_4= a e_5&, \; \nabla_{e_1} e_5=-a e_4, \; \nabla_{e_1} e_6= b e_7, \; \nabla_{e_1} e_7=-b e_7\\
	\nabla_{e_2} e_4=c e_5&, \; \nabla_{e_2} e_5=-c e_4, \; \nabla_{e_2} e_6=d e_7, \; \nabla_{e_2} e_7=-d e_6  \end{align*}
Plugging $\nabla_{e_i} e_i=0$ for all $1\leq i\leq 7$ into equation \eqref{div2} we have
\[\la \div(T_{\varphi}),e_j\ra_{\varphi}=- \sum_{i=1}^7 T_{\varphi}(e_i, \nabla_{e_i} e_j).\] For $1\leq j\leq 3$, it is clear that $\la \div(T_{\varphi}),e_j\ra_{\varphi}=0$. 

Let us compute now the other components.
\begin{align*} 
	\la \div(T_{\varphi}),e_4\ra_{\varphi}&=-T_{\varphi}(e_1,\nabla_{e_1} e_4)-T_{\varphi}(e_2,\nabla_{e_2} e_4)=-a T_{\varphi}(e_1, e_5)-c T_{\varphi}(e_2,e_5)\\
	\la \div(T_{\varphi}),e_5\ra_{\varphi}&=-T_\varphi(e_1, \nabla_{e_1} e_5)-T_\varphi(e_2,\nabla_{e_2} e_5)=a T_\varphi(e_1, e_4)+c T_\varphi(e_2, e_4)\\
	\la \div(T_{\varphi}),e_6\ra_{\varphi}&=-T_\varphi(e_1,\nabla_{e_1} e_6)-T_\varphi(e_2, \nabla_{e_2} e_6)=-b T_\varphi(e_1, e_7)-d T_\varphi(e_2,e_7)\\
	\la \div(T_{\varphi}),e_7\ra_{\varphi}&=-T_\varphi(e_1,\nabla_{e_1} e_7)-T_\varphi(e_2, \nabla_{e_2} e_7)=b T_\varphi(e_1, e_6)+d T_\varphi(e_2,e_6)
\end{align*} 

Recall that \[T_\varphi=\frac{\tau_0}{4} \la \cdot,\cdot\ra_\varphi-\star_\varphi(\tau_1\wedge \star_\varphi\varphi)-\frac{1}{2}\tau_2-\frac{1}{4} \jmath(\tau_3),\]
where $\jmath(\tau_3)(e_i,e_j)=\star_\varphi(\iota_{e_i} \varphi \wedge \iota_{e_j} \varphi \wedge \tau_3)$. Since $\{e_i\}_{i=1}^7$ is an orthonormal basis, the first term vanishes automatically. Note also from the formula for $\tau_2$ of Proposition \ref{torsionforms} that $\tau_2(e_1,e_j)=\tau_2(e_2,e_j)=0$ for $4\leq j\leq 7$.

We compute $\star_\varphi(\tau_1\wedge \star_\varphi\varphi)$ and $\jmath(\tau_3)$ (using the formulas from Proposition \ref{torsionforms}):
\begin{align*} 
	\tau_1 \wedge \star_\varphi\varphi&=-\frac{1}{6}(d+c) e^3 \wedge (-e^{1247}-e^{1256}-e^{1346}+e^{1357}+e^{2345}+e^{2367}+e^{4567})\\
	&=\frac{1}{6}(d+c) (e^{12347}+e^{12356}-e^{34567}).
	\end{align*}
Therefore \[\star_\varphi(\tau_1\wedge \star_\varphi\varphi)=-\frac{1}{6}(d+c)(e^{12}+e^{47}+e^{56}).\] Note that $\star_\varphi(\tau_1\wedge \star_\varphi\varphi)(e_1,e_j)=\star_\varphi(\tau_1\wedge \star_\varphi\varphi)(e_2,e_j)=0$ for $4\leq j\leq 7$.

Now, the interior products $\iota_{e_j} \varphi$, $1\leq j\leq 7$ are given by
\begin{align*} \iota_{e_1} \varphi=e^{23}+e^{45}+e^{67}&,\qquad \iota_{e_2} \varphi= -e^{13}+e^{46}-e^{57},\qquad 
	\iota_{e_3} \varphi=e^{12}-e^{47}-e^{56},\\
	\iota_{e_4} \varphi=-e^{15}-e^{26}+e^{37}&,\qquad \iota_{e_5} \varphi=e^{14}+e^{27}+e^{36},\qquad \iota_{e_6} \varphi=-e^{17}+e^{24}-e^{35},\\
	\iota_{e_7} \varphi=e^{16}-e^{25}-e^{34}&.
\end{align*} 
Hence, 
\[ \iota_{e_1} \varphi \wedge \tau_3=\frac{8}{7}(b+a)(e^{12345}+e^{12367}+e^{14567})+\frac{1}{2}(d+c)(e^{12346}-e^{12357}+2 e^{24567}),\]
\[ \iota_{e_2} \varphi \wedge \tau_3=\frac{1}{2}(d+c)(e^{12345}+e^{12367}+2e^{14567})+\frac{1}{7}(b+a)(e^{12346}-e^{12357}-6e^{24567}).\]
From this we have \[\jmath(\tau_3)(e_1,e_j)=\jmath(\tau_3)(e_2,e_j)=0,\quad \text{for} \;4\leq j\leq 7.\] 
In conclusion, $T_\varphi(e_1,e_j)=T_\varphi(e_2,e_j)=0$ for $4\leq j\leq 7$. Therefore, $\operatorname{div} T_\varphi=0$.
\end{proof}
\begin{obs}
	The 45 flat solvmanifolds of Table 1 can be obtained choosing values of $(a,b,c,d)$ such that $a\neq b$ and $c\neq d$. Indeed, instead of taking $A,B$ we take $A$ and $AB$, which corresponds to the values $(a,b),((a+c),(b+d))$. It can be easily deduced for the values of Table 1 that $a\neq b$ and $a+c\neq b+d$. This choice will induce an isomorphic lattice because $\la A,B\ra=\la A,AB\ra$. In this way, the invariant $G_2$-structure on the corresponding 45 flat solvmanifolds is divergence-free and it is a generic $G_2$-structure respect to Gray-Fernández classes, since none of the components of the torsion vanishes. 
\end{obs}

	\
	
\end{document}